\documentclass[11pt]{amsart}

\usepackage[top=30mm, bottom=30mm, left=30mm, right=30mm]{geometry}  

\usepackage{mathpazo}
\usepackage{amsmath,amssymb, amscd, color}
\usepackage{graphicx}
\usepackage{import}
\usepackage{amscd}
\usepackage{wrapfig}
\usepackage{epsfig}
\numberwithin{equation}{section}
\usepackage[alphabetic,nobysame]{amsrefs}
\usepackage{float}
\usepackage{tikz}
\usetikzlibrary{positioning}
\usepackage{comment}
\usepackage{cleveref}

\newcommand{\calE}{\mathcal{E}}
\newcommand{\calH}{\mathcal{H}}

\newcommand{\calK}{\mathcal{K}}
\newcommand{\calL}{\mathcal{L}}

\newcommand{\calT}{\mathcal{T}}

\newcommand{\HH}{\mathbb{H}}
\newcommand{\RR}{\mathbb{R}}

\newcommand{\NN}{\mathbb{N}}

\DeclareMathOperator{\leb}{Leb}

\DeclareMathOperator{\sh}{Sh}

\newcommand{\from}{\colon}

\newcommand{\bdy}{\partial}

\newtheorem{theorem}{Theorem}[section]

\newtheorem{lemma}[theorem]{Lemma}
\newtheorem{corollary}[theorem]{Corollary}
\newtheorem{proposition}[theorem]{Proposition}

\newtheorem{definition}[theorem]{Definition}

\newcommand{\teichmuller}{Teichm{\"u}ller{ }}

\makeatletter
 \let\c@theorem=\c@subsection
 \let\c@conjecture=\c@subsection
 \let\c@lemma=\c@subsection
 \let\c@proposition=\c@subsection
 \let\c@claim=\c@subsection
 \let\c@question=\c@subsection
 \let\c@criterion=\c@subsection
 \let\c@vfconj=\c@subsection
 \let\c@definition=\c@subsection
 \let\c@notation=\c@subsection
 \let\c@remark=\c@subsection
 \let\c@example=\c@subsection
 \let\c@equation=\c@subsection
 \let\c@figure=\c@subsection
 \let\c@wrapfigure=\c@subsection

\makeatother

\begin{document}
\title[Singular stationary measures cannot be quasi-symmetrically straightened]{Stationary measures on the circle from hyperbolic surfaces with cusps cannot be straightened by quasi-symmetries}

\author[Azemar]{Aitor Azemar}
\address{\hskip-\parindent
     School of Mathematics and Statistics\\
     University of Glasgow\\
     University Place\\
      Glasgow\\
      G12 8QQ United Kingdom}
\email{Aitor.Azemar@glasgow.ac.uk}
%\thanks{The first author acknowledges support from the GEAR Network
  %(U.S. National Science Foundation grants DMS 1107452, 1107263,
  %1107367 ``RNMS: GEometric structures And Representation varieties'').}

\author[Gadre]{Vaibhav Gadre}
\address{\hskip-\parindent
  School of Mathematics and Statistics\\
     University of Glasgow\\
     University Place\\
      Glasgow\\
      G12 8QQ United Kingdom}
\email{Vaibhav.Gadre@glasgow.ac.uk}
%\thanks{}

%\subjclass[2020]{ 30F60, 32G15, 37D40, 53D25, 37A25}

\begin{abstract}
Stationary measures on the circle that arise from a large class of random walks on the fundamental group of a finite-area complete hyperbolic surface with cusps are singular with respect to the Lebesgue measure. 
In particular, it is sufficient for singularity that a stationary measure satisfies an exponential decay for cusp excursions with excursions being measured in the path metric on horocycles bounding cusps.
In this note, we settle a conjecture of McMullen by proving that the singularity of stationary measures satisfying such exponential decay is quasi-symmetrically stable, that is under push-forward by any quasi-symmetry of the circle the measure remains singular.
\end{abstract}

%%%%%%%%%%%%%%%%%%%%%%%%%%%%%%%%%%%%%%%%%%%%%%%%

\maketitle

%%%%%%%%%%%%%%%%%%%%%%%%%%%%%%%%%%%%%%%%%%%%%%%%

\section{Introduction}
\label{s:intro}

Two intervals on the circle are said to be adjacent and equal if they share an endpoint and have the same length.
A circle map is said to be a \emph{quasi-symmetry} if there exists $K \geqslant 1$ such that for any pair $I$, $J$ of adjacent and equal intervals, $\leb(f(I)) \leqslant K \leb(f(J))$ and vice versa.
Suppose that $P$ is a property that subsets of the circle might satisfy.
We say that a particular set has \emph{quasi-symmetric stability for $P$} if its image under any quasi-symmetry satisfies $P$.

\medskip
%Quasi-symmetries of the circle are rich from many perspectives.
It is of interest to find explicit subsets that display quasi-symmetric stability.
By definition, a quasi-symmetry is H\"{o}lder continuous; hence it preserves zero Hausdorff dimension.
However, there are quasi-symmetries that promote a (Lebesgue) measure zero set to full measure. 
Specifically, Tukia defines a quasi-symmetry and a set with Hausdorff dimension less than one such that its image by the quasi-symmetry has its complement with Hausdorff dimension less than one. 
See \cite[Main theorem]{Tuk}.
By contrast, quasi-symmetries in higher dimensions are quasi-conformal and hence absolutely continuous. 

\medskip
With the circle being the boundary at infinity for the hyperbolic plane, singular measures on the circle (and measure zero sets that correspond to these) arise from various contexts in hyperbolic geometry. 
A quasi-symmetry of the circle extends to a quasi-isometry of the hyperbolic plane. 
Conversely, quasi-isometries of the hyperbolic plane give rise naturally to quasi-symmetries of the circle.
So it is natural to consider the behaviour of such measures and sets under quasi-symmetries. 
In particular, finite area complete hyperbolic surfaces/orbifolds with cusps have been used to give explicit measure zero sets that are quasi-symmetrically stable for various properties (such as measure zero, porosity, etc.).
In \Cref{s:stability}, we describe some of these sets and survey known results about them.
In this note, we expand the list by considering singular stationary measures arising from random walks on \emph{non-uniform lattices} in $SL(2,\RR) = \mathrm{Isom}(\HH^2)$ (that is, fundamental groups of finite-area complete hyperbolic surfaces/orbifolds with cusps). 
 
 \medskip
A random walk on a lattice $\Gamma$ in $SL(2,\RR) = \mathrm{Isom}(\HH^2)$ is non-elementary if the semigroup generated by its support contains a pair of independent hyperbolic isometries (with distinct stable/unstable fixed points on the circle at infinity).
Projected in to $\HH^2$ by the group action, a typical sample path converges to a point on $S^1 = \bdy_\infty \HH^2$.
The convergence defines a stationary measure on $S^1 = \bdy_\infty \HH^2$ giving the distribution of sample path limit points.

\medskip
A lattice $\Gamma$ is non-uniform when the quotient $X = \HH^2/ \Gamma$ is finite area but non-compact and so has cusps.
It is well known that stationary measures from a large class of random walks (for example, with finite word metric first moment) on a non-uniform lattice are singular.
See \cite{Gui-LeJ}, \cite{Der-Kle-Nav}, \cite{Bla-Hai-Mat}, \cite{Gad-Mah-Tio} for different approaches to understanding the singularity.
A lattice is uniform if the quotient is a closed orbifold. 
A long standing conjecture of Guivarc'h--Kaimanovich--Ledrappier states that stationary measures from finitely supported random walks on a uniform lattice are also singular.
See \cite{Der-Kle-Nav}.
Here, we restrict to non-uniform lattices.

\medskip
For a non-uniform lattice $\Gamma$, we fix cusp neighbourhoods in $X = \HH^2 / \Gamma$. 
The lifts of the cusp neighbourhoods give a horoball packing in $\HH^2$. 
We fix a base-point $x$ in $\HH^2$. 
Given a horoball $H$ and a hyperbolic ray $\gamma$ from $x$ that intersects $H$, we define the excursion of $\gamma$ in $H$ to be the distance (in the path metric on $\bdy H$) between the entry and exit points. 
Let $r \geqslant 0$ be a constant. 
We define the $r$-shadow of $H$ to be the interval at infinity given by hyperbolic rays from $x$ that have an excursion of size at least $r$ in $H$.  
We say that a random walk on $\Gamma$ has exponential decay for excursions if for any horoball $H$ and any sample path converging in to the shadow of $H$, the probability that it converges into the $r$-shadow decays exponentially in $r$.
The decay holds for finitely supported non-elementary random walks on $\Gamma$.
It is expected to be true for a larger class of random walks (for example with finite word metric first moment), but this is as yet unproven. 
We discuss the more general context for exponential decay in \Cref{s:exponential-decay} and sketch a proof for finitely supported random walks. 

\medskip
Our main result below settles a conjecture of McMullen. 

\begin{theorem}
\label{t:strong-singularity}

Suppose that $\Gamma$ is a non-uniform lattice in $SL(2,\RR)$. 
Suppose that $\mu$ is a non-elementary probability distribution on $\Gamma$.
Let $\nu$ be the stationary measure on $S^1$ arising from the $\mu$-random walk.
Suppose that $\nu$ has exponential decay for excursions. 
Then the pushforward $f^* \nu$ by any quasi-symmetry $f: S^1 \to S^1$ is singular with respect to the Lebesgue measure on $S^1$.
\end{theorem}

Douady--Earle show that a quasi-symmetry $f : S^1 \to S^1$ always extends to a bijective quasi-isometry $f$ of $\HH^2$. 
See \cite[Theorem 2]{Dou-Ear} and also \cite[Theorem 3.1]{Ibr} for a strengthening of it.
The pushforward measure $f^* \nu$ is then the stationary measure arising from the same random wlak on the group of quasi-isometries resulting from conjugating $\Gamma$ by $f$.
\Cref{t:strong-singularity} is then singularity for stationary measures arising from such groups of quasi-isometries.

\medskip
Let $S$ be a surface homeomorphic to $\HH^2/ \Gamma$. 
We may view the quotient $\HH^2/ \Gamma$ as defining a marked conformal structure on $S$ thus giving us a point in the corresponding \teichmuller space $\calT(S)$ of $S$.
When $f$ is $\Gamma$-equivariant then the conjugate $f^{-1} \Gamma f$ is again a non-uniform lattice in $SL(2,\RR)$ (isomorphic to $\Gamma$). 
Thus, it can be thought of as a different point of $\calT(S)$. 
By restricting to $\Gamma$-equivariant quasi-symmetries, \Cref{t:strong-singularity} implies the (already known) singularity of stationary measures given by all of $\calT(S)$ from singularity at a single point.

\medskip
For any finitely supported non-elementary random walk on a uniform lattice $\Gamma$ in $SL(2,\RR)$, random walk drift in the hyperbolic plane is a proper function on the \teichmuller space of the surface that is homeomorphic to $\HH^2 / \Gamma$. 
See \cite[Theorem A]{Aze-Gad-Gou-Hat-Les-Uya}.
It follows that for every such random walk on the abstract group there is a uniform lattice in $SL(2,\RR)$ such that the stationary measure that arises is singular.
If \Cref{t:strong-singularity} was true for uniform lattices then we would deduce the singularity conjectured by Guivarc'h--Kaimanovich--Ledrappier.
For reasons outlined in \Cref{s:dimensions}, we find it hard to speculate whether  \Cref{t:strong-singularity} holds for uniform lattices and if it does hold whether proving it is as hard as proving singularity itself.

\subsection{Quasi-symmetric stability:} 
\label{s:stability}

We list some known examples of quasi-symmetric stability.
\begin{itemize}
\item Because quasi-symmetries are H\"{o}lder continuous, any set with zero Hausdorff dimension remains so under any quasi-symmetry.
\item Any set that is a countable union of uniformly porous sets remains so under any quasi-symmetry.
See \cite[Theorem 4.2]{Vai}.
\item Any set that is strongly winning or absolutely winning remains so under any quasi-symmetry. 
See \cite[Theorem 2.2]{McM}.
\end{itemize}

\subsection{Diophantine sets}
Some sets arising from geometric considerations have been analysed extensively for quasi-symmetric stability for properties such as above.
The most known example is the set $D$ of Diophantine directions on a finite area complete hyperbolic surface $X$ with cusps. 
$D$ is the set of directions of hyperbolic rays from a base-point that have bounded excursions in the cusp neighbourhoods of $X$.
By ergodicity of the hyperbolic geodesic flow, it follows that $D$ has measure zero.
Using the concepts of games, McMullen shows that $D$ is an absolute winning set. 
McMullen also showed that any absolute winning set remains so under any quasi-symmetry.
See \cite[Theorem 1.3 and Theorem 2.2]{McM}. 
It is known that winning sets have maximal Hausdorff dimension; so $D$ has Hausdorff dimension one.
McMullen also proves that $D$ is countable union of uniformly porous sets.
See \cite[Proposition 1.5]{McM}.

\medskip
Porosity will be the key concept in our proof of \Cref{t:strong-singularity}.
So we give a proof sketch of the porosity result for $D$.
For a positive integer $m \geqslant 2$, we define $D_m \subset D$ consisting of hyperbolic rays whose excursions are all bounded above by $m$.
Given a horoball $H$, the $m$-shadow of $H$ is a subinterval of the shadow of $H$ with relative size (in the Lebesgue measure) $\approx 1/m$. 
The set $D_m$ avoids all these sub-intervals.
Using an estimate by Sullivan \cite{Sul} for the number of horoballs of a certain size in an interval, one consructs a horoball heirarchy where at each level we have that
\begin{itemize}
\item the horoball shadows have disjoint interiors and cover a set of full measure; and
\item for each horoball of the previous level, the shadows of the current level are contained in the $m$-shadow or in one of the components (in the horoball shadow) of the complement of the $m$-shadow. 
\end{itemize}
This shows that the set $D_m$ is $1/m$-porous.
It follows that $D$ is a countable union of porous sets, each uniformly porous.

\subsection{Stationary measures from non-uniform lattices}
For the stationary measures, a typical ray has unbounded excursions and a countable exhaustion by uniformly porous sets is not possible.
Using exponential decay of horoball excursions, we build a horoball heirarchy $\calH^{(n)}$ for which there is a constant $c> 0$ (that depends only on the random walk) such for a typical ray (for the stationary measure) its excursion in a horoball in $\calH^{(n)}$ is eventually (that is, for $n$ large enough depending on the ray) bounded above by $c \log n$. 
See \Cref{p:excursion-upper-bounds}.
This implies that eventually a typical ray misses sub-intervals of horoball shadows at all levels of the heirarchy.
The relative sizes of the sub-intervals in the corresponding shadows approach zero as the level goes to infinity but quantitatively this happens slowly.
This brings us to the notion of (non-uniform) porosity and we show that the support of the stationary measure is contained in a countable union of such porous sets.

\medskip
By a classical theorem of Ahlfors--Beurling (see \cite[Theorem 1]{Ahl-Beu}), any quasi-symmetry of the circle is given by integration with respect to a doubling measure.
By a theorem of Wu (see \cite[Theorem 1]{Wu}), non-uniform porous sets with a high enough porosity (slow enough decay of the relative sizes of complementary subintervals/holes) are measure zero for any doubling measure, hence quasi-symmetrically stable.
Our excursion upper bounds imply the slow enough decay of the relative sizes of avoided subintervals required in Wu's theorem. 
Our main theorem, namely \Cref{t:strong-singularity}, is then a direct consequence of Wu's theorem. 

\medskip
Singularity of stationary measures is also known in the following analogous contexts.
\begin{itemize}
\item Non-uniform lattices in the isometry group of $\HH^n$ for $n > 2$ (see \cite{Ran-Tio});
\item Mapping class groups of finite type surfaces with negative Euler characteristic (see \cite{Gad}, \cite{Gad-Mah-Tio2}).
\end{itemize}
The random walk boundaries in the above contexts are the spheres at infinity $S^{n-1} = \bdy_\infty \HH^n$ in the first case, and the Thurston boundaries of the corresponding \teichmuller spaces (see \cite{Kai-Mas}). 
The dimensions of these boundaries are at least two. 
Since quasi-symmetries (that is, quasi-conformal maps) in higher dimensions are absolutely continuous, quasi-symmetric stability holds directly.
However, the question of porosity is still of interest. 
%Our excursion upper bounds should hold in the above contexts.
We expect that the stationary measure has support contained in a countable union of highly porous (but not uniformly porous) subsets of the boundaries.

\subsection{Dimensions of stationary measures:}
\label{s:dimensions}

Our countable collection of sets that have high enough porosity is constructed using excursion upper bounds. 
In particular, (up to excising a countable subset of $S^1$) their union contains the set $D$ of Diophantine directions. 
Thus, the Hausdorff dimension of their union is one.

\medskip
On the other hand, the dimension of a singular stationary measure is expected to be strictly less than one thus giving us a stronger form of singularity.
See \cite[Conjecture 1.21]{Der-Kle-Nav} and in a more general form in \cite{Kai-LeP}.
For uniform lattices, finite first moment for a word metric is the same as finite first moment for the hyperbolic metric. 
So the dimension gap is conjectured for finitely supported random walks on uniform lattices, the general picture for diffused random walks being unclear.

\medskip
For non-uniform lattices, finite first moment in the word metric implies finite first moment for the hyperbolic metric but the converse is not true because the parabolic subgroups (corresponding to the cusps) are exponentially distorted.
Hence, the dimension gap may be conjectured for more general random walks that have finite word metric first moment. 
Our countable union of porous sets while serving the purpose for \Cref{t:strong-singularity} is nevertheless too large to give information regarding the dimension of the stationary measure.

\medskip
Despite the dimension conjecture for finitely supported random walks on uniform lattices, we think it to be likely that the stationary measure exhibits no porosity.
This should rule out our approach in proving \Cref{t:strong-singularity} for uniform lattices.
It is also possible that there exists a non-equivariant quasi-symmetry that does straighten the stationary measure in this case.

\subsection{Bounds on maximal excursions of typical geodesics:} 
\label{s:maximal-excursions}

Similar upper bounds hold when all cusp excursions of a typical hyperbolic ray for the stationary measure are considered.
Specifically, there is a constant $c > 0$ that depends only on the random walk such that along a typical ray the $n^{\text{th}}$ excursion $E_n$ is eventually (that is, for all $n$ large enough depending on the ray) at most $c \log n$.
This result follows directly from the exponential decay of excursions and the easy direction of the Borel--Cantelli lemma.
In fact, the upper bounds also hold in analogous contexts such as for non-uniform lattices in $\text{Isom}(\HH^n)$ for $n > 2$ and mapping class groups of finite type surfaces with negative Euler characteristic.

\medskip
Our construction of the horoball hierarchy is more involved because it is built to track porosity.
So it is not $\Gamma$-equivariant.
Excursions in these horoballs are only a subset of all cusp excursions.
Similarity of the bounds suggest that the set of excursions in the horoballs of the heirarchy has positive density among all excursions.

\medskip
The derivation of a logarithmic lower bound for the maximal excursion over all excursions up to time $T$ is harder and we do not attempt it here.
It should require the following two properties.
\begin{itemize}
\item A lower bound for the tail distribution (in the stationary measure) of an excursion.
\item A quasi-independence of excursions.
\end{itemize}
We expect these properties to be true when the random walk generates $\Gamma$ as a semi-group and is finitely supported (more generally, has a finite first word metric moment).
The arguments are likely to be quite technical. 
By knowing logarithmic lower bounds we would describe more precisely a full measure set for the stationary measure (in particular, it would not contain $D$). 
It might be sufficient to prove the dimension gap (referred to in \Cref{s:dimensions}).

\medskip
For the Lebesgue measure, the excursions are much larger. 
In particular, there is a constant $c' > 0$ such that for any $c> 1$ the largest excursion $E_{\max}(n)$ along any Lebesgue-typical ray is eventually (that is for $n$ large enough depending on the ray) at most $c' n (\log n)^c$. 
See \cite[Lemma 5.5]{Gad2}.
Lower bounds for $E_{\max}(n)$ are also known; for example Sullivan's $\log$-law says that the $\lim (\log E_{\max}(n)/ \log n)$ is $1/\log 2$. 
See \cite{Sul}.

\subsection{Acknowledgements} We thank C. McMullen for his comments on a preliminary draft of the paper. 

%%%%%%%%%%%%%%%%%%%%%%%%%%%%%%%%%%%%%%%%%%%%%%%%%%
\section{Horoball heirarchies}
\label{s:horoball-heirarchies}

Let $\Gamma$ be a non-uniform lattice in $SL(2, \RR)$.
Let $N_1, \cdots, N_k$ be (pairwise disjoint) cusp neighbourhoods in the finite area hyperbolic surface $X = \HH / \Gamma$. 
Lifting these neighbourhoods to $\HH$, let $\calH$ be the resulting horoball collection.
We call the complement $\HH - \cup_{H \in \calH} H$ the \emph{thick part} $\calT$.
%With the induced path metric, the thick part is quasi-isometric to the virtually free group $\Gamma$ (with a word metric with respect to a finite symmetric generating set) by the \u{S}varc-Milnor lemma.
We fix a base-point $x$ in $\calT$. 

\medskip
Let $p, p'$ be a pair of distinct points in $S^1 = \bdy_\infty \HH$. 
They partition $S^1$ into a pair of intervals with disjoint interiors.
Using the orientation on $S^1$ (compatible with orientation on $\HH$), we define $[p,p'] \subset S^1$ to be the interval that is positively oriented as we traverse from $p$ to $p'$.

\medskip
Let $H$ be any horoball in $\calH$ with $p_\infty(H)$ in $S^1$ its point at infinity. 
Let $\gamma^-(H)$ and $\gamma^+(H)$ be hyperbolic rays from $x$ with endpoints $p^-(H)$ and $p^+(H)$ in $S^1$ such that  
\begin{itemize} 
\item $\gamma^-(H)$ and $\gamma^+(H)$ and tangent to $H$; and
\item $p_\infty(H)$ is contained in $[p^-(H), p^+(H)]$. 
\end{itemize}
We call the interval $\sh_0(H) = [p^-(H), p^+(H)]$ the \emph{shadow at infinity} of $H$.

\medskip
Let $\gamma$ be a geodesic ray from $x$ that converges to a point in $\sh_0(H)$.
Let $x_p$ and $x'_p$ be the points of intersection of $\gamma$ with $\bdy H$. 
We define the \emph{excursion} $E(\gamma, H)$ to be the distance along $\bdy H$  between $x_p$ and $x'_p$.

\medskip
To define shadows more generally, suppose that $r \geqslant 0$.
Let $\gamma^-_r$ and $\gamma^+_r$ be geodesic rays from $x$ such that 
\begin{itemize}
\item $\gamma^-_r$ converges to a point $p^-_r$ in $[p^-(H), p_\infty(H)]$;
\item $\gamma^+_r$ converges to a point $p^+_r$ in $[p_\infty(H), p^+(H)]$; and
\item $E(\gamma^-_r, H) = r$ and $E(\gamma^+_r, H) = r$.
\end{itemize}
We define the shadow $\sh_r(H)$ to be the interval $[p^-_r , p^+_r]$. 
By definition, if $r < s$ then $\sh_s (H)$ is \emph{nested} in $\sh_r(H)$ that is, the closure of $\sh_s(H)$ is contained in the interior of $\sh_r (H)$.
Let $\leb$ denote the Lebesgue measure on $S^1$.
It follows from the nesting (and basic hyperbolic geometry) that there is a constant $c> 0$ such that for any horoball $H$ in $\calH$
\begin{equation}
\label{e:proportion}
\leb(\sh_r(H)) \geqslant \frac{c}{r} \leb(\sh_0(H))
\end{equation}
See \cite[Section 3]{Gad2}.

\medskip
We organise horoballs by the (Euclidean) sizes of their shadows as follows.
Suppose that $0 < \rho < 1$ is a positive constant.
We say that a horoball has $\rho$ size $n$ if $\rho^{n+1} < \leb(\sh_0 (H) )\leqslant \rho^n$.
Suppose that $I$ is an interval in $S^1$.
We define $\calH_n(I)$ to be the collection of those horoballs $H$ in $\calH$ such that 
\begin{itemize}
\item $\sh_0(H)$ is contained in $I$; and
\item the $\rho$-size of $H$ is $n$.
\end{itemize}

\medskip
We recall \cite[Proposition 4]{Sul}. 
The proof uses the mixing of the hyperbolic geodesic flow on $T^1 \HH^2 / \Gamma$. 
The proposition is not stated as precisely in \cite{Sul} as it is stated here.
However, our statement is implicit in the proof of \cite[Proposition 4]{Sul} requiring no modification other than tracking explicitly how the counting depends on the compact subset of the boundary.

\begin{proposition}
\label{p:horoballs-of-a-certain-size}
Suppose that $\Gamma$ is a non-uniform lattice in $SL(2,\RR)$. There are positive constants $C$ and $\rho < 1$ such that for any interval $I$ in $S^1$ there is a positive integer $n_I$ such that 
\[
| \calH_n(I) | \geqslant C \frac{\leb(I)}{\rho^n}
\]
for all $n \geqslant n_I$.
\end{proposition}

Suppose that $I$ is an interval in $S^1$. 
We define the subcollection $\calH'_n(I) \subset \calH_n(I)$ by requiring that for any pair of horoballs $H_1$ and $H_2'$ in $\calH'_n(I)$ the shadows $\sh_0(H_1)$ and $\sh_0(H_2)$ have disjoint interiors.
We use \Cref{p:horoballs-of-a-certain-size} to deduce:

\begin{lemma}
\label{l:defininte-proportion}
Suppose that $\Gamma$ is a non-uniform lattice in $SL(2,\RR)$. There is a positive constant $0 < \delta < 1$ such that for any interval $I$ in $S^1$ there is a positive integer $n$ such that the union $U = \cup_{H \in \calH'_n(I)} \sh_0(H)$ satisfies 
\[
\leb(U) > \delta \leb(I)
\]
\end{lemma}

\begin{proof}
By \Cref{p:horoballs-of-a-certain-size}, we may fix $n_1 = n_I$ and consider horoballs $H$ such that $\sh_0(H)$ is contained in $I$ and the $\rho$-size of $H$ is $n_1$.
By definition, each horoball shadow can intersect at most $1/\rho$ shadows of size $n$.
Thus, we may retain at least $C' \leb(I)/\rho^n$ horoballs (where $C' = c\rho$) to get the collection $\calH'_n(I)$.
It follows that for $U = \cup_{H \in \calH'_n(I)} \sh_0(H)$ we have 
\[
\leb(U) = \sum_{H \in \calH'_n(I)} \sh_0(H) \geqslant (|\calH'_n(I)| ) \cdot \rho^{n+1} = C \rho^2  \leb(I)
\]
The lemma follows by choosing $\delta < C \rho^2$. 
\end{proof}

We use \Cref{l:defininte-proportion} to deduce:
\begin{lemma}
\label{l:horoball-shadows-fill-interval}
Suppose that $I$ is an interval in $S^1$.
There exists a collection $\calH (I) \subset \calH$ of horoballs $H$ such that 
\begin{itemize}
\item for distinct horoballs $H$ and $H'$ in $\calH (I) $, the interiors of $\sh_0(H)$ and $\sh_0(H')$ are disjoint; and
\item the union $U  = \bigcup_{H \in \calH (I)} \sh_0(H)$ has full $\leb$-measure in $I$.
\end{itemize}
\end{lemma}

\begin{proof}
By \Cref{l:defininte-proportion} applied to $I$, we can find a collection $\calH(I, 1) = \calH'_n(I)$ of horoballs all of $\rho$-size $n$ such that their shadows have disjoint interiors and their union $U(I,1)$ satisfies $\leb(U(I,1)) > \delta \leb(I)$.
We now consider the intervals that are the components of $I - U(I,1)$. 
These are finitely many and again by  \Cref{l:defininte-proportion} applied to each of these intervals, we can find a collection of horoballs for each interval such that their shadows have disjoint interiors and their union has proportion (for the Lebesgue measure) at least $\delta$ in the interval.
We may then define $\calH (I, 2)$ to be the collection of these horoballs and define $U(I,2)$ to be the union of their shadows. 
Note that $\leb(U(I, 2)) > \delta \leb(I - U(I,1))$.

\medskip
Iterating infinitely, we get a sequence horoball collections $\calH (I,k)$ and subsets $U(I,k)$ of $I$ such that
\begin{itemize}
\item $U(I, 0)$ is the empty set, 
\item $U(I,k) \subset I - \cup _{i \leqslant k-1} U(I,i) $; 
\item for $k \geqslant 1$ we have that $U(I,k)$ is the finite union of  shadows of horoballs in $\calH (I,k)$ with pairwise disjoint interiors, and
\item there is a constant $\delta > 0$ such that $\leb(U (I,k) ) > \delta \leb(I - \cup _{i \leqslant k-1} U (I, i) )$ for all $k \geqslant 1$. 
\end{itemize}
It follows that the infinite union $\cup_k  U (I, k)$ is a countable union of horoball shadows with pairwise disjoint interiors and it has full measure in $I$.
\end{proof}

Using \Cref{l:horoball-shadows-fill-interval}, we define the first sub-collection $\calH^{(1)}$ of horoballs in our hierarchy.

\begin{lemma}
\label{l:first-collection}
There exists a collection $\calH^{(1)} \subset \calH$ of horoballs $H$ such that 
\begin{itemize}
\item for distinct horoballs $H$ and $H'$ in $\calH^{(1)}$, the interiors of $\sh_0(H)$ and $\sh_0(H')$ are disjoint; and
\item the union $U^{(1)}  = \bigcup_{H \in \calH^{(1)}} \sh_0(H)$ has full $\leb$-measure in $S^1$.
\end{itemize}
\end{lemma}

\begin{proof}
The proof follows directly from \Cref{l:horoball-shadows-fill-interval} applied to $S^1$. 
\end{proof} 

\medskip
Let $\alpha > 0$ be a constant.
The horoball sub-collections of $\calH$ that we define below depend on $\alpha$.
We drop the dependence from our notation to keep it simple.

\medskip
We define the sequence $r_n = \alpha \log n$.
We now define the collection $\calH^{(n)}$ for $n \geqslant 2$ inductively. 
Let $H$ be a horoball in $\calH^{(n-1)}$. 
Let $\calK$ be the collection of horoballs $H'$ such that $\sh_0(H')$ is contained in either $\sh_{r_n}(H)$ or in the closure of one of the components of $\sh_0(H) - \sh_{r_n}(H)$. 
By \Cref{l:horoball-shadows-fill-interval}, we can find a collection $\calH^{(n)}_H$ of horoballs in $\calK$ such that 
\begin{itemize}
\item for distinct horoballs $K$ and $K'$ in $\calH^{(n)}_H$, the interiors of $\sh_0(K)$ and $\sh_0(K')$ are disjoint; and
\item the union $\bigcup_{K \in \calH^{(n)}_H} \sh_0(K)$ has full measure in $\sh_0(H)$. 
\end{itemize}
See \Cref{f:horoballs}.

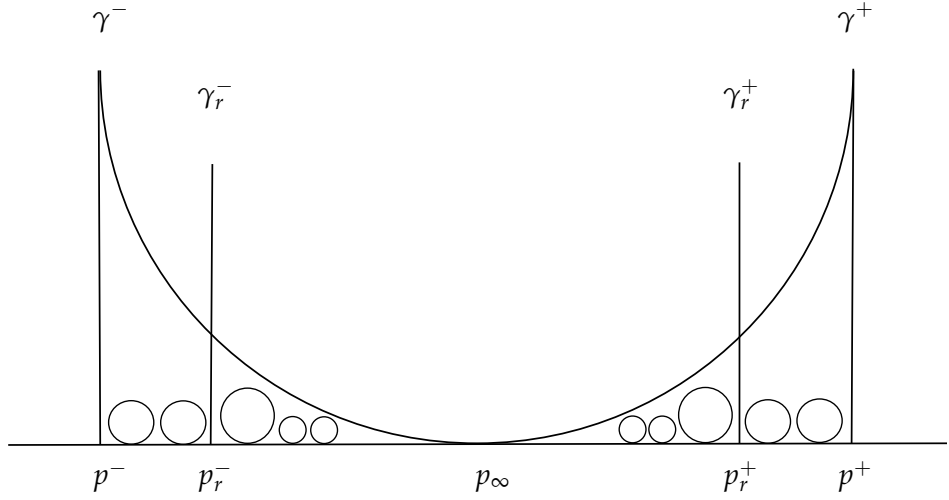
\begin{figure}
\centering
\includegraphics[scale=0.5]{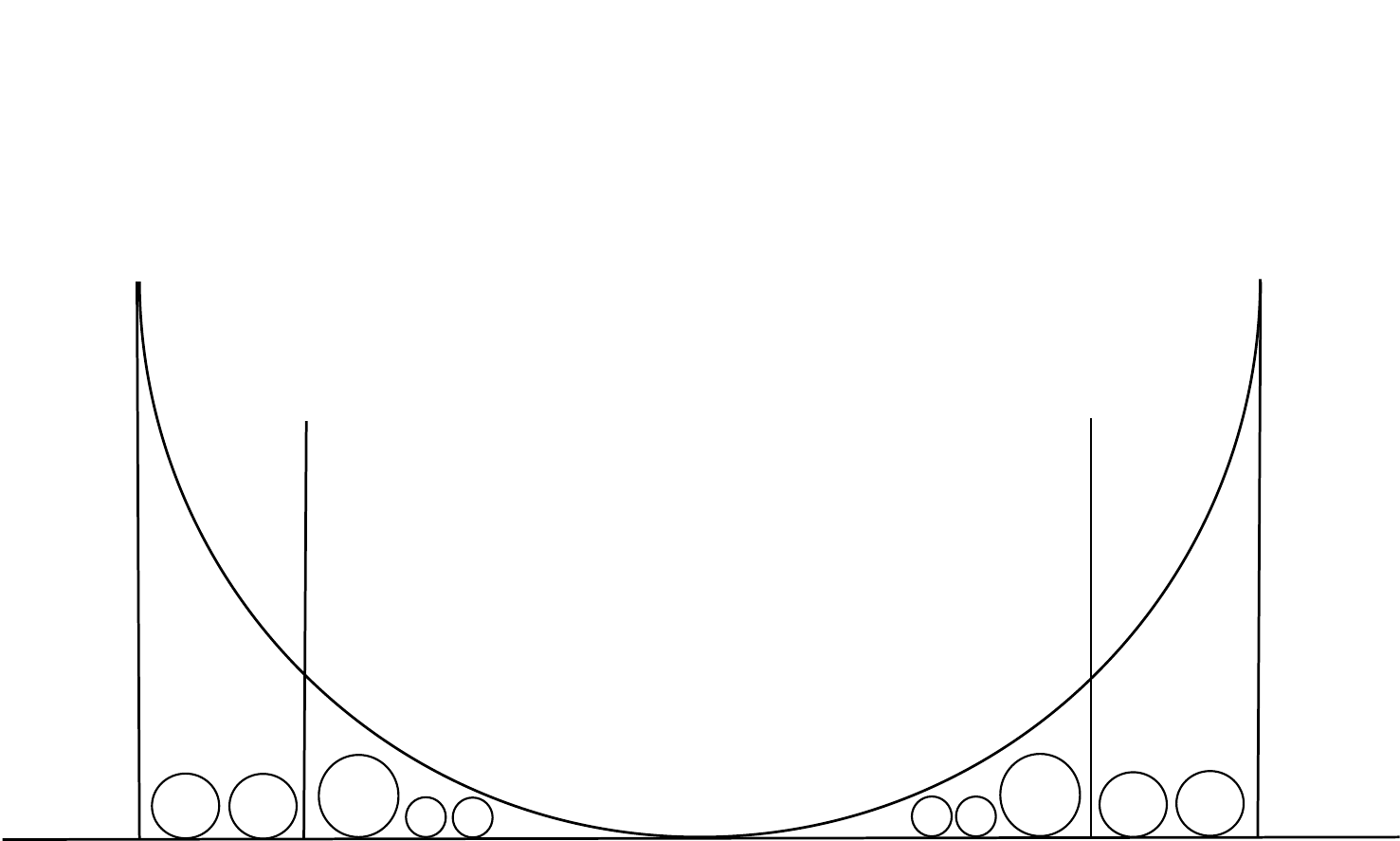}
\setlength{\unitlength}{1in}
\begin{picture}(0,0)(0,0)
\put(-2.55,-0.2){$p_\infty $}
\put(-4.55,-0.2){$p^-$}
\put(-0.65,-0.2){$p^+$}
\put(-4,-0.2){$p^-_r$}
\put(-1.25,-0.2){$p^+_r$}
\put(-4.55,2.2){$\gamma^-$}
\put(-4,1.8){$\gamma^-_r$}
\put(-1.25,1.8){$\gamma^+_r$}
\put(-0.65,2.2){$\gamma^+$}
\end{picture}
\vskip 10pt
\caption{Illustration of how the collection $\calH^{(n)}_H$ is constructed.} 
%Part of the loop $\gamma$ is depicted as a solid curve. 
%The dotted lines represent the boundaries of the polytopes. 
%Unlike the boxes $V_{k-1}$ and $V_{k+1}$, the box $V_k$ is not contained inside a polytope, so the diagonal flow must be applied to it. 
%The resulting segment $\delta_k$ is shown as dashed curve.}
\label{f:horoballs}
\end{figure}

We define $\calL^{(n)}_H$ to be the collection of those horoballs $H'$ in $\calH^{(n)}_H$ such that $\sh_0(H')$ is contained in $\sh_{r_n}(H)$.
We define $L^{(n)}_H$ to be the union $\bigcup_{H' \in \calL^{(n)}_H} \sh_0(H')$.
By \Cref{e:proportion},  
\[
\leb \left( L^{(n)}_H \right) \geqslant \frac{c}{r_n} \leb( \sh_0(H)) 
\]

\medskip
We define 
\[
\calH^{(n)} = \bigcup_{H \in \calH^{(n-1)}} \calH^{(n)}_H
\]
Note that the union $U^{(n)} = \bigcup_{H \in \calH^{(n)}} \sh(H) $ is full $\leb$-measure in $S^1$.

\medskip
We also define 
\[
\calL^{(n)} = \bigcup_{H \in \calH^{(n-1)}} \calL^{(n)}_H
\]
We define $L^{(n)}$ to be the union $\bigcup_{H \in \calL^{(n)}} \sh_0(H)$.
Since $U^{(n)}$ has full measure, by \Cref{e:proportion} we get 
\[
\leb \left( L^{(n)} \right) \geqslant \frac{c}{r_n} \leb \left( U^{(n)} \right) = \frac{c}{r_n}
\]

%%%%%%%%%%%%%%%%%%%%%%%%%%%%%%%%%%%%%%%%%%%%%%
\section{Porous Sets}
\label{s:porous}

We fix $n \geqslant 1$.
We define the set $Y(n, n+1)$ by
\begin{equation}
\label{e:Y(n,n+1)}
Y(n, n+1) = U^{(n)}  - L^{(n)} 
\end{equation}
Let $m > n$. 
We define the set $Y_{n, m}$ inductively by $Y(n, m) = Y(n, m-1) - L^{(m-1)} \cap Y(n, m-1)$. 
We define $X_n$ to be the infinite intersection $\bigcap_{m > n} Y(n,m)$. 

\begin{definition} 
\label{d:porous}
Let $a_n$ be a sequence of numbers satisfying $0 < a_n < 1$. 
A set $X$ is said to be $\{a_n \}$-porous if there is a sequence of coverings $\calE^{(n)} = \{ E(n,j) \}$ by intervals with disjoint interiors such that
\begin{itemize}
 \item $E(n,j) - X$ contains an interval $E'(n,j)$ such that $\leb(E'(n,j) \geqslant a_n \leb(E(n,j))$,
 \item $\bigcup_{\calE^{(n+1)}} E(n+1, k)$ is contained in $\bigcup_{\calE^{(n)}} E(n,j) - E'(n,j)$, and
 \item $\sup_{\calE^{(n)}}  \leb(E(n,j))$ tends to $0$ as $n \to \infty$.
 \end{itemize} 
\end{definition}
We call the intervals $E'(n, j)$ the \emph{holes at level $j$}.

\begin{proposition}
\label{p:porous}

The set $X_n$ is porous for the sequence 
\[
a_j = \frac{c}{r_{n+j}}
\]
\end{proposition}

\begin{proof}
For each $j \geqslant 1$, we set $\calE^{(j)}$ to be collection of shadows $\sh_0(H)$ where $H$ is contained in $Y(n, n+j-1)$.
For each such shadow, we set the hole to be $\sh_{r_{n+j}}(H)$. 
By \Cref{e:proportion}, the required estimate for the relative measure follows. 
\end{proof} 

%%%%%%%%%%%%%%%%%%%%%%%%%%%%%%%%%%%%%%%%%%%%%%%%%
\section{Quasi-symmetries and Measure zero sets}
\label{s:null}

We say that two intervals $I$ and $J$ in $S^1$ are \emph{adjacent and equal} if they share an endpoint and have equal length. 
\begin{definition}
\label{d:quasi-symmetry}

A map $f:S^1 \to S^1$ is said to be a \emph{quasi-symmetry} if there exists a constant $K \geqslant  1$ such that for any adjacent and equal intervals $I$ and $J$ in $S^1$ we have
\[
\frac{1}{K} \leqslant \frac{\leb(f(I)) }{\leb(f(J))} \leqslant K
\] 
\end{definition}

\begin{definition}
\label{d:doubling}

A measure $\mu$ on $S^1$ is said to be a \emph{doubling} measure if there exists a constant $M> 0$ such that for any $p \in S^1$ and any radius $R > 0$, the interval $I(p, 2R)$ centred at $p$ and with radius $2R $ satisfies
\[
0 < \mu(I(p, 2R)) < M \mu(I(p, R))
\]
 \end{definition}

By a classical result of Ahlfors--Beurling, namely \cite[Theorem 1]{Ahl-Beu},  a map $f: S^1 \to S^1$ is quasi-symmetric if and only if it is given by integration with respect to a doubling measure, that is, if there exists a doubling measure $\mu$ and $\theta > 0$ such that 
\[
f(p') - f(p) = \int_p^{p'} d \mu
\] 

We begin by recalling \cite[Theorem 1]{Wu} 

\begin{theorem}
\label{t:Wu}

Suppose that $a_n$ is a sequence of numbers satisfying $0 < a_n < 1$. 
Suppose that for all real numbers $\kappa \geqslant 1$ the sum 
\[
\sum_{n=1}^\infty (a_n)^\kappa
\]
diverges. 
Then any $\{a_n\}$-porous set has measure zero for any doubling measure.
\end{theorem}

In other words, if the sequence $\{ a_n \}$ satisfies the hypothesis then for any quasi-symmetric map $f: S^1 \to S^1$ the image of an $\{a_n \}$-porous set has measure zero for the Lebesgue measure.

\medskip
Recall our definition of the set $X_n$. 
By \Cref{t:Wu}, we deduce:

\begin{proposition}
\label{p:measure-zero-for-qs}

Suppose that $f : S^1 \to S^1$ is a quasi-symmetry. We have
\[
\leb( f(X_n) ) =0
\]
\end{proposition}

\begin{proof}
Note that for any $\kappa \geqslant 1$, the sequence $a_j$ in \Cref{p:porous} satisfies 
\[
\sum_j  a_j^\kappa  = \sum_j  \left( \frac{c}{\alpha \log (n+j)} \right)^\kappa = \infty
\]
The proposition follows from~\Cref{t:Wu}. 
\end{proof}

%%%%%%%%%%%%%%%%%%%%%%%%%%%%%%%%%%%%%%%%%%%%%%
\section{Stationary measures and horoball excursions}
\label{s:excursions-for-stationary-measure}

\subsection{Horoball excursions and nested sets in the thick part}
The action of the (virtually free) group $\Gamma$ on the thick part $\calT$ is co-compact.
By  the \u{S}varc-Milnor lemma, the thick part $\calT$ with the induced path metric $D$ is quasi-isometric to $\Gamma$ for any word metric on $\Gamma$ defined by a finite symmetric generating set.
Let $\pi \from \HH \to \calT$ be the closest point projection. 
 We call the projection $\pi(\gamma)$ of a hyperbolic geodesic $\gamma$ a \emph{projected path}.

\medskip
Let $x$ be a point in $\calT$ and $H$ a horoball with the point at infinity $p_\infty(H) \in S^1$. 
With $x$ as a base-point, let $\gamma_\infty (H)$ be the geodesic ray from $x$ to $p_\infty(H)$.
Recall that $\gamma^- (H)$ and $\gamma^+ (H)$ are rays from $x$ tangent to $H$.
Also recall that $\gamma^-_r$ and $\gamma^+_r$ are rays from $x$ that have excursions of size $r$ in $H$.

\medskip
Let $s^-_r$ denote the arc $\pi(\gamma^-_r) \cap \partial H$.
Let $x^-_1$ and $x^-_2$ be the endpoints of $s^-_r$ with $x^-_r$ the farthest (among the two) from $x$. 
For any geodesic ray $\gamma$ from $x$ such that $\gamma^-_r, \gamma, \gamma_\infty$ are in the counter-clockwise order, the projected path $\pi(\gamma)$ contains the arc $s^-_r$. 
For such a ray, let $t^-_r$ denote the initial segment of $\pi(\gamma)$ from $x$ to $x^-_2$. 
Let $R^-(r)$ be the set of such rays and define the \emph{thick shadow}
\[
C^- (x, H; r) = \bigcup_{\gamma \in R^-(r) } \pi(\gamma) - t^-_r
\]

\medskip
We similarly define $s^+_r$ and use it to define the thick shadow $C^+ (x, H; r)$ using the ray $\gamma^+_r$.

\medskip
We define $\bdy_\infty C^- (x, H; r)$ to be the set of endpoints at infinity of the rays in $R^-(r)$. 
It follows that $\bdy_\infty C^- (x, H; r) = [p^-_r (H) , p_\infty (H)]$ and $\bdy_\infty = [p_\infty (H), p^+_r (H)]$ and so their union is $\sh_r (H)$. 

\begin{lemma}
\label{l:nesting} 
There exists constants $K \geqslant 1$ and $r_0 > 0$ that depend only on $\Gamma$ and the choice of cusp neighbourhoods such that for all $r$ large enough and $(r'  - r)/K - r_0 > 0$, we have 
\begin{equation}\label{e:nesting}
D (\calT - C^-(x, H; r), C^-(x, H; r')) = \frac{r' - r}{K} - r_0,
\end{equation}
\end{lemma}

\begin{proof}
By \cite[Lemma 2.1]{Gad-Mah-Tio}, a projected path is a quasi-geodesic in $\calT$.
Since $C^- (x, H;R)$ is defined with a projected arc as a prefix, the lemma follows.
\end{proof}

\subsection{Exponential decay for cusp excursions}
\label{s:exponential-decay}

Let $\mu$ be a non-elementary probability distribution on $\Gamma$.
We consider the random walk on $\Gamma$ given by $\mu$.
We call the random product $\omega_n = g_1 \cdots g_n$ (where each $g_i$ is sampled independently by $\mu$) a sample path of length $n$. 
Sample paths of length $n$ are distributed by the $n$-fold convolution $\mu^{(n)}$. 
The {\em path space for the random walk} is the space $\Gamma^{\NN}$ of one-sided infinite sequence in $\Gamma$ equipped with the push-forward of the product measure under the map $(g_1, g_2, \cdots) \to (\omega_1, \omega_2, \cdots)$. 
We denote this measure as $\nu$ and it has marginals $\mu^{(n)}$.
The group $\Gamma$ acts on the left while the random walk increments $g_n$ act by right multiplication $\omega_{n-1} \to \omega_n$. 

\medskip
Let $x$ be any base-point in $\HH$. 
For $\nu$-almost every infinite sample path $\omega^+$, the sequence $\omega^+_n  x$ in $\HH$ converges to a point $p = p(\omega)$ in $S^1$. 
We may then pushforward $\nu$ to get a measure on $S^1$ that gives the distribution of $p$.
In a slight abuse of notation, we continue to denote the measure on $S^1$ by $\nu$. 

\medskip
We now define what it means for a random walk to have exponential decay for excursions. 

\begin{definition}
\label{d:good-excursion-decay}

A random walk has \emph{exponential decay for excursions} if there exists constants $c \geqslant 0$ and $\alpha > 0$ such that for any horoball $H$ in $\calH$ and all constants $r, s \geqslant 0$ such that $ s \geqslant r+ c $, we have 
\[
\nu  (\sh_s(H)) \leqslant e^{-\alpha} \nu (\sh_r(H)).
\]
\end{definition}

\medskip
Note that the exponential decay implies that for $r \geqslant c$ and $\beta = \alpha((1/c) - (1/r))$ 
\[
\nu (\sh_r (H)) \leqslant e^{-\beta r} \nu (\sh_0(H))
\]

\medskip
For completeness, we include a proof sketch that finitely supported random walks that generate the lattice as a semi-group have the desired exponential decay for excursions.
As a special case of more general exponential decay of shadows (see \cite[Lemma 2.10]{Mah}), the exponential decay is true for any finitely supported non-elementary random walk.
In the specific contexts mentioned here, we expect the exponential decay to be true for all non-elementary random walks with finite word metric first moment, but this is open at present.

\begin{proof}[Proof sketch for finitely supported random walks that generate $\Gamma$:]
For any pair of points $z \neq z'$ in $\calT$, let $H(z, z')$ be the half-space (in the metric $D$ on $\calT$) of points closer to $z'$ than $z$.
Let $\bdy_\infty H(z, z')$ be the limit set at infinity of $H(z,z')$. 

\medskip
Suppose that $H$ is a horoball.
Suppose that $\xi > 0$ is a constant. 
Suppose that $r' = r + 3 \xi$. 
Recall that $x \in \calT$ is our base-point. 
We consider the thick shadows $C^- (x, H; r') \subset C^- (x, H; r)$ and similarly $C^+ (x, H; r') \subset C^+(x, H; r)$. 
We define the set $C^- (x, H; r + \xi) - C^- (x, H; r+2\xi)$ to be the \emph{midpoint set} for the pair $C^- (x, H; r') \subset C^- (x, H; r)$. 
We denote the set by $M^- (\xi)$. 
Similarly, we define the set  $C^+ (x, H; r + \xi) - C^+ (x, H; r+2\xi)$ to be the midpoint set for the pair $C^+ (x, H; r') \subset C^+(x, H; r)$. 
We denote it by $M^+ (\xi) $.

\medskip
By \Cref{l:nesting}, there is $\xi' > 0$ such that for any $\xi \geqslant \xi'$ there is $u =  u(\xi) > 0$ such that such that for any point $x'$ in the midpoint set $M^- (\xi)$ (similarly $M^+ (\xi)$) there is a point $x''$ with $D (x',x'') = u$ such that $\bdy_\infty H(x',x'')$ is contained in the interior of $[p^-_r (H) , p^-_{r'} (H)]$ (similarly $[p^+_{r'} (H), p^+_r (H)]$).

\medskip
Since the random walk generates $\Gamma$, there is a constant $\kappa > 0$ such that for any $z$ such that $D(x, z) = u$, we have $\nu (\bdy_\infty H(x, z) ) > \kappa$. 

\medskip
Let $\delta$ be maximum of $d_\calT (x_0, g x_0)$ over $g$ in support of $\mu$. 
Suppose that $H$ is a horoball. 
Recall the constant $r_0$ from~\Cref{e:nesting}. 
Let $\xi = \max \{ \xi', K(\delta + r_0) \}$.
Suppose that $r' > r$ be such that $r' = r + 3 \xi$. 
Then any sample path that converges to $\sh_{r'}(H)$ intersects $M^- (\xi)$ or $M^+ (\xi)$. 
Conditioning on the point of intersection, we conclude that the probability for a sample path starting at this point to converge in to $\sh_r (H) - \sh_{r'} (H) $ is at least $\kappa$.
The claimed decay for $\nu$ now follows. \qedhere
\end{proof}

\subsection{Upper bounds for excursions in the horoball heirarchy}
\label{s:excursion-upper-bounds}

We now return to the horoball hierarchies $\calH^{(n)}$  constructed in  \Cref{s:horoball-heirarchies}.
Suppose that $p$ lies in $S^1$.
Suppose that $\gamma$ is the hyperbolic ray from the base point $x$  to $p$.
If $\gamma$ has an excursion in a horoball $H$ in $\calH^{(n)}$ then we denote the excursion by $E^{(n)}(\gamma)$. 
By definition, for $n \geqslant 2$, the set $L^{(n)}$ is the set of points $p$ in $S^1$ such that the ray $\gamma$ from $x$ to $p$ satisfies $E^{(n-1)}(\gamma) \geqslant r_n$. 

\medskip
We prove:

\begin{proposition}
\label{p:excursion-upper-bounds}

There exist a constant $\alpha > 0$ that depends only on $\mu$ such that for the horoball collections $\calH^{(n)}$ and for $\nu$-almost every $p \in S^1$ there is a positive integer $N = N(\gamma)$ such that with $r_n = \alpha \log n$, we have 
\[
E^{(n-1)}(\gamma) < r_n
\]
for all $n > N$. 
\end{proposition}

\begin{proof}
We fix $\alpha > 0$ to satisfy $\alpha > 1/ \beta$. 
We then consider horoball heirarchies $\calH^{(n)}$ corresponding to the sequence $r_n = \alpha \log n$. 
By the exponential decay of excursion from  \Cref{d:good-excursion-decay}, it follows that the
\[
 \nu \left( L^{(n)} \right) < e^{-\beta r_n} = e^{-\beta \alpha \log n} = \frac{1}{n^{\beta \alpha}}
 \]
 It follows that 
\[
\sum_{n \geqslant 2}  \nu(L^{(n)}) < \infty
\]
By the Borel--Cantelli lemma (the easy direction), it follows that $\nu( \limsup L^{(n)}) = 0$. 
The conclusion follows. 
\end{proof} 

We deduce:

\begin{corollary}
\label{c:support}
We have 
\[
\nu \left( \bigcup_{n \geqslant 1} X_n  \right) = 1
\]
\end{corollary}

\begin{proof}
Suppose that $p$ is $\nu$-typical and suppose that $\gamma$ is the hyperbolic ray from $x$ to $p$.
By \Cref{p:excursion-upper-bounds}, there exists a positive number $N$ such that $E^{(n-1)}(\gamma) < r_n$ for all $n > N$. 
Thus, $p$ lies in $X_N$ and we get our conclusion.
\end{proof}

\begin{proof}[Proof of \Cref{t:strong-singularity}]
Set $X = \cup_{n \geqslant 1} X_n$.
Suppose that $f : S^1 \to S^1$ is a quasi-symmetry.
Consider $f(X)$.
By \Cref{p:measure-zero-for-qs}, we have
\[
\leb (f(X)) \leqslant \sum_{n \geqslant 1} \leb(f(X_n)) = 0
\]
On the other hand, by definition of a pushforward measure and \Cref{c:support}
\[
f^* \nu(f(X)) = \nu (X) = 1
\]
Thus $f^\ast \nu$ is singular, as desired.
\end{proof}


\begin{thebibliography}{99}

\bibitem{Ahl-Beu} Ahlfors, L. and Beurling, A. 
\newblock {\em The boundary correspondence under quasi-conformal mappings}, Acta. Math. 96 (1956) 125-142.

\bibitem{Aze-Gad-Gou-Hat-Les-Uya} Azemar, A., Gadre, V., Gou\"{e}zel, S., Haettel, T., Lessa, P. and Uyanik, C. 
\newblock  {\em Random walk speed is a proper function on \teichmuller space}, to appear in J. Mod. Dyn.

\bibitem{Bla-Hai-Mat} Blach\`{e}re, S., Ha\"{i}ssinsky, P. and Mathieu, P. 
\newblock {\em Harmonic measures versus quasi-conformal measures for hyperbolic groups}, Ann. Sci. \'{E}c. Norm.  Sup\'{e}r. (4) 44 (2011), no.  4, 683-721.

\bibitem{Coo} Cooper, D. 
\newblock {\em Quasi-isometries of hyperbolic spaces are almost isometries.} Proc. Amer. Math. Soc 123 (1995), no. 7, 2221-2227. 

\bibitem{Der-Kle-Nav} Deroin, B., Kleptsyn, V. and Navas, A. 
\newblock {\em On the question of ergodicity of minimal group actions on the circle}, Mosc. Math. J. 9 (2009), no. 2, 263-303.

\bibitem{Dou-Ear} Douady, A. and Earle, C-J. 
\newblock{\em Conformally natural extension of homeomorphisms of the circle}, Acta Math. 157 (1986), no. 1-2, 23-48.

\bibitem{Gad} Gadre, V. 
\newblock {\em Harmonic measures for distributions with finite support on the mapping class group are singular}, Duke Math J. 163 (2014), no. 2, 309-368.

\bibitem{Gad2} Gadre, V.
\newblock {\em Partial sums of excursions along random geodesics and volume asymptotics for thin parts of moduli spaces of quadratic differentials}, J. Eur. Math. Soc. 19 (2017), no. 10, 3053-3089.

\bibitem{Gad-Mah-Tio} Gadre, V., Maher, J. and Tiozzo, G. 
\newblock {\em Word length statistics and Lyapunov exponents for Fuchsian groups with cusps}, New York J. Math 21 (2015), 511-531.

\bibitem{Gad-Mah-Tio2} Gadre, V., Maher, J. and Tiozzo, G. 
\newblock{Word length statistics and singularity of harmonic measure}, Comment. Math. Helv 92 (2017), no. 1, 1-36.

\bibitem{Geh-Vai} Gehring, F. and V\"{a}is\"{a}l\"{a}, J.  
\newblock {\em Hausdorff dimension and quasi-conformal mappings}, J. London Math. Soc. (2) 6 (1973), 504-512.

\bibitem{Gui-LeJ} Guivarc'h, Y. and Le Jan, Y. 
\newblock {\em Asymptotic winding of geodesic flow on modular surfaces and continued fractions}, Ann. Sci. \'{E}c. Norm.  Sup\'{e}r. (4) 26 (1993),  no. 1, 23-50.

\bibitem{Ibr} Ibragimov, Z. 
\newblock{\em Quasi-isometric extensions of quasi-symmetric mappings of the real line compatible with composition}, Ann. Acad. Sci. Fenn. Math. 35 (2010), no. 1, 221-233.

\bibitem{Kai-LeP} Kaimanovich, V. and LePrince, V. 
\newblock{Matrix random products with singular harmonic measure}, Geom. Dedicata 150 (2011), 257-279.

\bibitem{Kai-Mas} Kaimanovich, V and Masur, H. 
\newblock{The Poisson boundary of the mapping class group}, Invent. Math. 125 (1996), 221-264.

\bibitem{Mah} Maher, J. 
\newblock {\em Exponential decay in the mapping class group}, J. Lond. Math. Soc. (2) 86 (2012), no. 2, 366-386.

\bibitem{McM} McMullen, C. 
\newblock {\em Winning sets, quasi-conformal maps and Diophantine approximation}, Geom. Funct. Anal. 20 (2010), 726-740.

\bibitem{Ran-Tio} Randecker, A. and Tiozzo, G. 
\newblock{Cusp excursion in hyperbolic manifolds and singularity of harmonic measure}, J. Mod. Dyn. 17 (2021), 183-211.

\bibitem{Ser} Series, C. 
\newblock {\em The modular surface and continued fractions.} J. London Math. Soc. (2) 31 (1985), no. 1, 69-80.

\bibitem{Sul} Sullivan, D. 
\newblock{\em Disjoint spheres, approximation by imaginary quadratic numbers and the logarithmic law for geodesics}, 149 (1982), no. 3-4, 215-237.

\bibitem{Tuk} Tukia, P. 
\newblock {\em Hausdorff dimension and quasi-symmetric mappings}, Math. Scand. 65 (1989), no. 1, 152-160.

\bibitem{Vai} V\"{a}is\"{a}l\"{a}, J. 
\newblock {\em Porous sets and quasi-symmetric maps}, Trans. Amer. Math. Soc. 299 (1987), no. 2, 525-533.

\bibitem{Wu} Wu, J-M. 
\newblock {\em Null sets for doubling and dyadic doubling measures}, Ann. Acad. Sci. Fenn. Ser. A I Math. 18 (1993), no. 1, 77-91.

\end{thebibliography}
\end{document}